\documentclass[12pt]{amsart}
\usepackage{amscd}
\usepackage{amsmath}
\usepackage{amssymb}
\usepackage{units}
\usepackage[all]{xy}
\def\frk{\frak}               

\def\mm{{\frk m}}

\def\Phi{{\frk n}}
\def\Phi{{\frk N}}
\def\opn#1#2{\def#1{\operatorname{#2}}} 
\opn\chara{char} \opn\length{\ell} \opn\pd{pd} \opn\rk{rk}
\opn\projdim{proj\,dim} \opn\injdim{inj\,dim} \opn\rank{rank}
\opn\depth{depth} \opn\sdepth{sdepth} \opn\fdepth{fdepth}
\opn\grade{grade} \opn\height{height} \opn\embdim{emb\,dim}
\opn\codim{codim}  \opn\min{min} \opn\max{max}

\opn\Tr{Tr} \opn\bigrank{big\,rank}
\opn\superheight{superheight}\opn\lcm{lcm}
\opn\trdeg{tr\,deg}
\opn\reg{reg} \opn\lreg{lreg} \opn\ini{in} \opn\lpd{lpd}
\opn\size{size}
\opn\div{div} \opn\Div{Div} \opn\cl{cl} \opn\Cl{Cl}
\opn\Spec{Spec} \opn\Supp{Supp} \opn\supp{supp} \opn\Sing{Sing}
\opn\Ass{Ass} \opn\Min{Min}
\opn\Ann{Ann} \opn\Rad{Rad} \opn\Soc{Soc}
\opn\Im{Im} \opn\Ker{Ker} \opn\Coker{Coker} \opn\Am{Am}
\opn\Hom{Hom} \opn\Tor{Tor} \opn\Ext{Ext} \opn\End{End}
\opn\Aut{Aut} \opn\id{id}  \opn\deg{deg}

\opn\nat{nat}
\opn\pff{pf}
\opn\Pf{Pf} \opn\GL{GL} \opn\SL{SL} \opn\mod{mod} \opn\ord{ord}
\opn\Gin{Gin} \opn\Hilb{Hilb}
\opn\aff{aff} \opn\con{conv} \opn\relint{relint} \opn\st{st}
\opn\lk{lk} \opn\cn{cn} \opn\core{core} \opn\vol{vol}
\opn\link{link} \opn\star{star}
\opn\gr{gr}

\def\pot#1#2{#1[\kern-0.28ex[#2]\kern-0.28ex]}

\opn\dirlim{\underrightarrow{\lim}}
\opn\inivlim{\underleftarrow{\lim}}
\let\to=\rightarrow

\def\Implies{\ifmmode\Longrightarrow \else
        \unskip${}\Longrightarrow{}$\ignorespaces\fi}
\def\implies{\ifmmode\Rightarrow \else
        \unskip${}\Rightarrow{}$\ignorespaces\fi}
\def\iff{\ifmmode\Longleftrightarrow \else
        \unskip${}\Longleftrightarrow{}$\ignorespaces\fi}

\let\:=\colon
\newtheorem{Theorem}{Theorem}[]

\newtheorem{Corollary}[Theorem]{Corollary}
\newtheorem{Proposition}[Theorem]{Proposition}
\newtheorem{Remark}[Theorem]{Remark}

\newtheorem{Question}[Theorem]{Question}

\let\epsilon\varepsilon
\let\phi=\varphi
\let\kappa=\varkappa
\def\qed{\ifhmode\textqed\fi
      \ifmmode\ifinner\quad\qedsymbol\else\dispqed\fi\fi}
\def\textqed{\unskip\nobreak\penalty50
       \hskip2em\hbox{}\nobreak\hfil\qedsymbol
       \parfillskip=0pt \finalhyphendemerits=0}
\def\dispqed{\rlap{\qquad\qedsymbol}}

\opn\dis{dis}
\def\pnt{{\raise0.5mm\hbox{\large\bf.}}}

\opn\Lex{Lex}

\begin{document}
\title{  Nested  Artin Strong Approximation Property.}

\author{Zunaira Kosar and Dorin Popescu }
\thanks{}

\address{Zunaira Kosar, Abdus Salam School of Mathematical Sciences,GC University, Lahore, Pakistan}
\email{ zunairakosar@gmail.com}

\address{Dorin Popescu, Simion Stoilow Institute of Mathematics of the Romanian Academy, Research unit 5,
University of Bucharest, P.O.Box 1-764, Bucharest 014700, Romania}
\email{dorin.popescu@imar.ro}

\begin{abstract} We study the Artin Approximation property with constraints in a different frame. As  a consequence we  give a nested Artin Strong Approximation property for algebraic power series rings over a field.


\end{abstract}

\maketitle

\vskip 0.5 cm

\section*{Introduction}

Let $K$ be a field and $R=K\langle x\rangle$, $x=(x_1,\ldots, x_n)$ be the ring of algebraic power series in $x$ over $K$, that is the algebraic closure  of the polynomial ring $K[x]$ in the formal power series ring ${\hat R}=K[[x]]$.
 Let $f=(f_1,\ldots,f_q) $ be a system of polynomials in $Y=(Y_1,\ldots,Y_p)$ over $R$ and $\hat y$ be a solution of $f$ in the completion ${\hat R}$ of $R$.
\begin{Theorem}[M. Artin \cite{A}]\label{ar}  For any $c\in {\bf N}$ there exists a solution $y^{(c)}$ in $R$ such that $y^{(c)}\equiv {\hat y}$  mod $(x)^c$.
\end{Theorem}

Also  M. Artin proved before (see \cite{A0}) a similar statement for the ring $R$ of complex convergent power series and later  (see \cite[p.7]{A1}) asked, whether, given $c\in \mathbb{N}$ and a formal solution $ y(x)=( y_1(x), \ldots, y_p(x)) \in \mathbb{C}[[x]]^p$ satisfying
$$ y_j(x)\in K [[x_1, \ldots, x_{s_j}]] \ \ \forall j\in [p]$$
for some integers $s_j\in [n]$, there exists a convergent solution $ y'(x)$ of $f$
such that $y'(x)\equiv y(x)$ mod $(x)^c$ and
 $$y'_j(x)\in k\{x_1, \ldots,x_{s_j}\} \ \ \forall j\in [p].$$

Shortly after,  A. Gabrielov \cite{Ga} (see also \cite[Example 5.3.1]{R}) gave an example showing that the answer to the previous question is negative in general.
On the other hand, since Theorem \ref{ar} remains valid if we replace convergent power series by algebraic power series,  the question of M. Artin is also relevant in this context and it appears that in this case the question has a positive answer as it is shown in  \cite[Theorem 3.7]{P} (see also \cite[Corollary 3.7]{P1}, \cite[Theorem 5.2.1]{R}).

\begin{Question}\label{q1}
 (Artin Approximation with constraints \cite[Problem 1, page 65]{R}) Let $R$ be an excellent local subring of $K[[x]]$, $x=(x_1,\ldots,x_n)$ such that the completion of $R$ is $K[[x]]$ and $f\in R[Y]^q$, $Y=(Y_1,\ldots,Y_p)$. Assume that there exists a formal solution $\hat{y}\in K[[x]]^p$ of $f=0$ such that $\hat{y}_i\in  K[[\{x_j: j\in J_i\}]]$ for some subset $J_i\subset [n]$, $i\in [p]$. Is it possible to approximate $\hat{y}$  by a solution $y\in R^p$ of $f=0$ such that
$y_i\in R\cap K[[\{x_j: j\in J_i\}]]$, $i\in [p]$?
\end{Question}
Similarly,  we considered below the following question.

\begin{Question}\label{q2} Let $K\subset K'$ be a field extension, $R=K[x]_{(x)}$, $R'=K'[x]_{(x)}$ (resp.  $R=K\langle x\rangle$, $R'=K'\langle x\rangle$) and $f\in R[Y]^p$, $Y=(Y_1,\ldots,Y_m)$.
Assume that there exists a  solution $\hat{y}\in R'^p$ of $f=0$ such that $\hat{y}_i\in  K'[\{x_j: j\in J_i\}]_{(x_{J_i})}$ (resp. $ K'\langle\{x_j: j\in J_i\}\rangle$) for some subset $J_i\subset [n]$, $i\in [m]$. Is it possible to find  a solution $y\in R^p$ of $f=0$ such that
$y_i\in  K[\{x_j: j\in J_i\}]_{(x_{J_i})}$ (resp,  $ K\langle\{x_j: j\in J_i\}\rangle$), and  $\ord y_i=\ord \hat{y}_i$, $i\in [p]$?
\end{Question}

We show (see Proposition \ref{proploc.} and Theorem \ref{t}) that Question \ref{q2} has a positive answer when the field extension $K\subset K'$ is algebraically pure.
A ring morphism $u:A\to B$ is called {\em algebraically pure} (see \cite{P0}), if every finite system of polynomial equations over $A$ has a solution in $A$ if it has a solution in $B$. A finite type ring morphism is algebraically pure if and only if it has a retraction and a filtered inductive limit of algebraically pure morphisms is an algebraically pure morphism by \cite{P0} (see also \cite[Theorem 1.10]{Pi}). It is easy to see that a field extension of an algebraically closed field is algebraically pure and the ultrapower of fields define algebraically pure field extension (see the proof of Theorem   
 \ref{t2}).

The above questions are related with the so called  Artin approximation property. There exists also a strong approximation property (see \cite{Gr}, \cite{A}, \cite{PP}, \cite{K},
 \cite{P0}, \cite{P}, \cite{P1}, \cite{R}, \cite{P2}).

\begin{Question} (Strong Artin Approximation with constraints \cite[Problem 2, page 65]{R}) \label{q3}
Let us consider $f(y)\in K[[x]][y]^q$ and $J_i \subset [n]$, $i \in [p]$. Does there
exist a function $\nu : \mathbb{N} \to \mathbb{N}$ such that
for all $c \in \mathbb{ N}$ and all $\hat{y}_i(x) \in  K[[x_{J_i}]]$, $j \in J_i$, $i \in [p]$, such that
$f(\hat{y}(x))\in (x)^{\nu(c)}$, $\hat{y}(x)=(\hat{y}_1(x),\ldots,\hat{y}_p(x))$,
there exist $y_i(x) \in K[[x_{J_i}]]$, $i \in [m]$ such that $f(y(x)) = 0$, $y(x)=(y_1(x),\ldots,y_p(x))$ and $\hat{y}_i(x) \equiv y_i(x)$ mod $ (x)^c$,
$i \in [ p]$?
\end{Question}

This question has a positive answer when $K=\mathbb{C}$ but a negative one when $K=\mathbb{Q}$ (see \cite{BDLV}, \cite[Proposition 3.3.4]{R} and \cite[Example 5.4.5]{R})). 
Similarly,  we considered below the following nested question.

\begin{Question}\label{q4}

Let us consider a field $K$,   $f=(f_1,\ldots,f_q)\in K\langle x\rangle[Y]^q$, $Y=(Y_1,\ldots,Y_p)$  and  $0\leq s_1\leq \ldots \leq s_p\leq n$ be some non-negative integers.
 Does there
exist a function $\nu : \mathbb{N}^p \to \mathbb{N}$ such that
for all $c \in \mathbb{ N}^p$ and all $\hat{y}_i(x) \in  K[x_1,\ldots,x_{s_i}]$,  such that $\ord \hat{y}_i(x)=c_i$,  $i \in [p]$ and
$f(\hat{y}(x))\in (x)^{\nu(c)}$,  $\hat{y}(x)=(\hat{y}_1(x)\ldots,\hat{y}_p(x))$,
there exist $y_i(x) \in K\langle x_1,\ldots,x_{s_i}\rangle$, $i \in [p]$ such that $f(y(x)) = 0$, $y(x)=(y_1(x),\ldots,y_p(x))$ and $\ord y_i(x)=c_i$,
$i \in [ p]$?
\end{Question}

Theorem \ref{t2} shows a positive answer to this question.
The proof uses the ultrapower methods (see \cite{BDLV}, \cite{P0}, \cite{P}, \cite{P1}) and so it is not constructive. We should mention that there exists  a stronger result (see Remark \ref{r1}).  

Artin approximation property with constraints is necessary in CR Geometry (see \cite{BM} and \cite{Mir}), the nested case appears in the construction of the analytic deformations of a complex analytic germ when it has an isolated singularity (see \cite{Gra}, 
\cite{K}, \cite{P2}). It is also used to prove that analytic set germs are homeomorphic to set germs (see \cite{Mo}), or to polynomial germs (see \cite{BPR}) having no assumption on the singular locus. 

We owe thanks to G.\ Rond who hinted us the Remark \ref{r1} and to a referee who showed us several misprints and had useful comments.

\section{Properties on polynomial rings similar to the Artin approximation with constraints.}

Question \ref{q1} together with \cite[Theorem 3.4]{CPR} give the idea of the following proposition.

\begin{Proposition}
Let $M \subset K[x]^p$ be a finitely generated $K[x]$-submodule and $K \to K'$  a field extension. Let $J_i$, $i\in [p]$ be  subsets of $[n]$, $x_{J_i}=(x_k)_{k\in J_i}$ and $A_i=K[x_{J_i}]$,  $A'_i=K'[x_{J_i}]$, $i\in [p]$. Set
 $\mathcal{N}=M \cap (A_1 \times \cdots \times A_p)$ and
  $\mathcal{N'}=(K'[x]M)\cap(A'_1 \times \cdots \times A'_p).$
If $\hat{u}=\sum_{j=1}^{r}w_j\hat{v}_j \in \mathcal{N'}$ for some $w_j \in M$ and $\hat{v}_j \in K'[x]$ then there exist $v_{j}\in K[x]$ such that
$u=\sum_{j=1}^{r}w_jv_{j} \in \mathcal{N}$ and $\min_{i\in [p]}\ord u_i=\min_{i\in [p]}\ord \hat{ u}_i$.
\end{Proposition}

\begin{proof}
Let $w_1, \ldots , w_r$ be the generators of $M$ and ${\hat u}\in \mathcal{N'}$. We have \begin{equation}\hat{u}=\sum_{j=1}^{r}w_j\hat{v}_j\end{equation}\label{eq1} has this form for some polynomials $\hat{v}_j\in K'[x]$.
Let $\hat{u}_i=\sum_{l\in \mathbb{N}^n, |l|\leq \alpha}\hat{u}_{il}x^l$,  $w_{ij}=\sum_{t\in \mathbb{N}^n, |t|\leq \beta}w_{i,j,t}x^t$, $i\in [p]$ and $\hat{v}_j=\sum_{q\in \mathbb{N}^n,|q|\leq \gamma}\hat{v}_{j,q}x^q$, where $\alpha=\beta+\gamma$.
The components of equation (\ref{eq1}) give a system of $(p(\sum_{1\leq s\leq \beta}{s+n-1\choose s}))$-linear equations as
 $$T\tilde{v}=\tilde{u},$$
 \noindent where $T$ is a $(p(\sum_{1\leq s\leq \beta}{s+n-1\choose s})) \times (r(\sum_{1\leq s'\leq \gamma}{s'+n-1\choose s'}))$ matrix of entries of coefficients of $w_j$ from $K$, $\tilde{v}$ is a vector of entries from the coefficients of $\hat{v}_j$ from $ K'$ and $\tilde{u}$ is a vector of entries from the coefficients of $\hat{u}_i$ from $ K'$. Now this is a system $L$ of linear equations with coefficients in $K$ and has a solution $(\hat{v}_{jq}), (\hat{u}_{il})$ in $K'$. We may consider in $L$ only variables corresponding to those $(\hat{v}_{jq}), (\hat{u}_{il})$ which are not zero. Since $K \to K'$ is a flat morphism  there exists a solution of $L$ in $K$ say $(v_{jq}), (u_{il})$. Thus we have $u=\sum_{l\in \mathbb{N}^n, |l|\leq \alpha}{u}_lx^l \in K[x]^p$, and ${v}_j=\sum_{q\in \mathbb{N}^n, |q|\leq \gamma}{v}_{jq}x^q \in K[x]$ such that $u=\sum_{j=1}^{r}w_jv_j$.\\

Assume that $\min_{i\in [p]}\ord \hat{ u}_i=\ord \hat{u}_{i_0}=c$ for some $i_0\in [p]$. Then there exists $\nu_1,\ldots,\nu_n\in \mathbb{N}$ such that $\sum_i\nu_i=c$ and  $\hat{u}_{i_0\nu} \neq 0$. If we have $u_{i_0\nu}=0$ for all solutions of $L$ from $K$ then we get a contradiction because $(\hat{u}_{il})$ is generated by the solutions of $L$ from $K$. Hence
$u_{i_0\nu}\not =0$. Since the coefficients of $u_i$ are zero when the corresponding coefficients of $\hat{ u}_i$ are zero we see that $u\in \mathcal{N}$ and  it  has no coefficients corresponding to ${\hat u}_{il} $ when $|l|<c$. Hence we have $\min_i\ord u_i=\ord u_{i_0}=c$.
\hfill\ \end{proof}

\begin{Proposition}\label{proploc.}
Let $K \to K'$ be an algebraically pure morphism of fields and $x=(x_1, \ldots , x_n)$. Let $J_i$, $i\in [p]$ be  subsets of $[n]$, $x_{J_i}=(x_k)_{k\in J_i}$ and $A_i=K[x_{J_i}]_{(x_{J_i})}$,  $A'_i=K'[x_{J_i}]_{(x_{J_i})}$, $i\in [p]$. Set
 $\mathcal{N}= A_1 \times \cdots \times A_p$ and
  $\mathcal{N'}=A'_1 \times \cdots \times A'_p$. Let $f$ be a system of polynomials from $K[x]_{(x)}[Y]$, $Y=(Y_1,\ldots,Y_p)$, and ${\hat y}\in \mathcal{N'}$,  such that $f({\hat y})=0$.
Then there exists $y\in \mathcal{N}$  such that $f(y)=0$ and   $\ord y_i=\ord \hat y_i$ for $i\in [p]$.
\end{Proposition}
\begin{proof}
Let $f \in K[x]_{(x)}[Y]$ be of the form $$f(Y)=\sum_{j\in \mathbb{N}^n,|j|< \alpha}w_{j}Y^j,$$ where $w_{j}$ are the coefficients of $f$ in $Y$ and belong to $K[x]_{(x)}$. Since ${\hat y}\in \mathcal{N'}$ is a solution of $f$,  we get
\begin{equation}\sum_{j\in \mathbb{N}^n,|j|< \alpha}w_{j}\hat{y}^j=0\end{equation}\label{eq2}
Let $\hat{y}=\frac{\hat{P}(x)}{\hat{Q}(x)}\in \mathcal{N'}$, where $\hat{P}(x)=(\hat{P}_1(x),\ldots,\hat{P}_p(x))$, $\hat{Q}(x)=(\hat{Q}_1(x),\ldots,\hat{Q}_p(x))$, $\hat{P}_i(x), \hat{Q}_i(x)\in K'[x_{J_i}] $,  and $\hat{Q}_i(x)\notin (x_{J_i})$ for all $i\in [p]$. Let $w_j=(w_{1j},\ldots,w_{pj})$, $w_{ij}=\frac{P'_{ij}(x)}{Q'_{ij}(x)}\in K[x]_{(x)}$, where $j\in \mathbb{N}^n,|j|< \alpha$, $Q'_{ij}(x)\in K[x]$ and $Q'_{ij}(x) \notin (x)$.
 Let $\mathbf{Q}'(x)$ be the least common multiple of $((Q'_{ij})(x))$ for all $j\in \mathbb{N}^n,|j|< \alpha$ and $\mathbf{\hat{Q}}(x)=  \Pi_{i\in [p]}(\hat{Q}_i(x))$. If we multiply the equation (2) by $\mathbf{Q}'(x)\mathbf{\hat{Q}}(x)^{\alpha}$ we get
 \begin{equation}\mathbf{Q}'(x)\mathbf{\hat{Q}}(x)^{\alpha}(\sum_{j\in \mathbb{N}^n,|j|< \alpha}w_{j}\hat{y}^j)=0  .\end{equation} \label{eq3}
 From equation (\ref{eq3}) we see that $f$ will be transformed in a different system of equations in $\hat{P}_i, \hat{Q}_i$, that is we change $f$ by another system in more $\hat{y}$ (in fact $2p$ will be the new $p$)  but this time they are from $K'[x_{J_1}] \times \ldots \times K'[x_{J_p}]$ for some bigger $p$ and we look for a solution from  $K[x_{J_1}] \times \ldots \times K[x_{J_p}]$. This new system will give a system of equations $F$ in the nonzero  coefficients $(\hat{y}_{iq})$ of $\hat{y}_i $ in $x$.
 Moreover we add for each $i,q$ with $\hat{y}_{iq}\not = 0$ a new equation $G_{iq}=Y_{iq}Y_{iq}^{'-1}-1$ which has a solution in $K'$ given by $\hat{y}_{iq}, \hat{y}_{iq}^{-1}$.
Thus $F,G=(G_{iq})$  must have a solution $(y_{iq}), (y'_{iq})$  also in $K$. Then
$$\{q\in \mathbb{N}^n:y_{iq}\not =0\}=\{q\in \mathbb{N}^n:{\hat y}_{iq}\not =0\}$$
for any $i\in [p]$.
 It follows that the new $y$ given by $y_i=\sum_q y_{iq}x^q$ satisfies $y\in K[x_{J_1}] \times \ldots \times K[x_{J_p}]\subset \mathcal{N}$,
 $f(y)=0$ and $\ord y_i=\ord\ \hat{y}_i$ for all $i\in [p]$.
\hfill\ \end{proof}

\begin{Corollary}\label{corloc.}
Let $K$ be an algebraically closed field, $K \subset K'$  a  field extension and $x=(x_1, \ldots , x_n)$. Let $J_i$, $i\in [p]$ be  subsets of $[n]$, $x_{J_i}=(x_k)_{k\in J_i}$ and $A_i=K[x_{J_i}]_{(x_{J_i})}$,  $A'_i=K'[x_{J_i}]_{(x_{J_i})}$, $i\in [p]$. Set
 $\mathcal{N}= A_1 \times \cdots \times A_p$ and
  $\mathcal{N'}=A'_1 \times \cdots \times A'_p$. Let $f$ be a system of polynomials from $K[x]_{(x)}[Y]$, $Y=(Y_1,\ldots,Y_p)$, and ${\hat y}\in \mathcal{N'}$,  such that $f({\hat y})=0$.
Then there exist $y\in \mathcal{N}$  such that $f(y)=0$ and   $\ord y_i=\ord \hat y_i$ for $i\in [p]$.
\end{Corollary}
For the proof note that a field extension of an algebraically closed  field is  an algebraically pure  field extension (see \cite[Corollary 1.8]{P0}).

\section{Properties on algebraic power series similar to the Artin approximation with constraints.}

\begin{Theorem}\label{t}
Let $K \to K'$ be an algebraically pure morphism of fields and $x=(x_1, \ldots , x_n)$. Let $J_i$, $i\in [p]$ be  subsets of $[n]$, $x_{J_i}=(x_k)_{k\in J_i}$ and $A_i=K\langle x_{J_i}\rangle$, resp. $A'_i=K'\langle x_{J_i}\rangle$, $i\in [p]$ be the algebraic power series in $x_{J_I}$ over $K$ resp. $K'$. Set
 $\mathcal{N}= A_1 \times \cdots \times A_p$ and
  $\mathcal{N'}=A'_1 \times \cdots \times A'_p$. Let $f$ be a system of polynomials from $K\langle x\rangle [Y]$, $Y=(Y_1,\ldots,Y_p)$, and ${\hat y}\in \mathcal{N'}$,  such that $f({\hat y})=0$.
Then there exist $y\in \mathcal{N}$  such that $f(y)=0$ and   $\ord y_i=\ord \hat y_i$ for $i\in [p]$.
\end{Theorem}
\begin{proof}
As $\hat y_i\in A'_i$ and $A'_i$ is a filtered inductive limit of etale neighborhoods of $K'[x_{J_i}]_{(x_{J_i})}$ we may find $y'_i$ in an etale $K'[x_{J_i}]$-algebra $U_i$ such that $\hat{y}_i$ is the image of $y'_i$ by the limit map $\rho_i:U_i\to K'\langle x_{J_i}\rangle$. Let us say $y'_i\in U_i=
  (K'[x_{J_i}][T_i]/(F_i))_{G_i}$, where $F_i$ is monic in $T_i$ and $G_i$ is a multiple of $\partial F_i/\partial T_i$.  The images of  $\rho_i$ are contained in an etale neighborhood of $K'[x]_{(x)}$ and so they factor through an etale $K'[x]$-algebra, let us say through
$U= (K'[x][T]/(F))_G$,  where $F$ is monic in $T$ and $G$ is a multiple of $\partial F/\partial T$.
 Suppose that the $K'[x_{J_i}]$-morphism $\varphi_i:U_i\to U$
is given by $T_i\to a'_i\in U$. It follows that $F_i(a'_i)=0$ and $G_i(a'_i) $ is invertible in $U$, that is

i) $G^rF_i(a'_i)\in (F)$,

ii) $G_i(a'_i)| G^r$ modulo $(F)$

\noindent for some $r\in \mathbb{N}$.

Note that the coefficients of $f$ belong to an etale neighborhood of $K[x]_{(x)}$  and so we may consider the coefficients of $f$ as images of some elements of an etale $K[x]$-algebra $V$ by the limit map $\theta:V\to K\langle x\rangle$. Let us say $V=(K[x,\tilde T]/(\tilde F))_{\tilde G}$,  where $\tilde F$ is monic in $\tilde T$ and $\tilde G$ is a multiple of $\partial \tilde F/\partial \tilde T$. More precisely, assume that $f=\sum_{j< \beta}w_{j}Y^j$,  $w_j\in K\langle x\rangle$ and choose $w'_j\in V$ with $\theta (w'_j)=w_j$. Set  $f'=\sum_{j< \beta}w'_{j}Y^j\in V[Y]$.
We may enlarge $U$ such that the composite map $V\to K\langle x\rangle\to
 K'\langle x\rangle$ factors through $U$, let us say the $K[x]$-morphism $\psi :V\to U$ is given by $\tilde T\to b'\in U$.

Let $\rho:U\to K'\langle x\rangle$ be the corresponding limit map.
  We have the following commutative diagram.

 $$
  \begin{xy}\xymatrix{K[x]  \ar[d]\ar[r] & V \ar[d]^{\psi} \ar[r]^{\theta} & K\langle x\rangle \ar[d]\\
K'[x] \ar[r] & U \ar[r]^{\rho} & K'\langle x\rangle \\
K'[x_{J_i}]\ar[u]\ar[r]  & U_i  \ar[u]_{\phi_i} \ar[r]^{\rho_i} &K'\langle x_{J_i}\rangle
 \ar[u] } \end{xy}
  $$

 We have

 i') $G^r\tilde F(b')\in (F)$,

ii') $\tilde G(b')| G^r$ modulo $(F)$

\noindent possibly by increasing  $r$. By multiplying $f'$ by a certain power of $\tilde G$ we may consider that the coefficients of $f'$ are in $K[x,\tilde T]$.  Changing again $U$ if necessary, we have \\
$f'(\psi(\tilde T),\varphi_1(y'_1),\ldots, \varphi_p(y'_p))=0$ in $U$. Replacing
 $y'_i$ by $\bar{y}'_i/G_i^{r_i}$ modulo $F_i$ for some $r_i\in {\mathbb{N}}$,
 $\bar{y}'_i\in K'[x_{J_i},T_i]$ we may change $f'$ into a system of polynomials $\tilde f$ in $Y, Y'=(Y'_1,\ldots,Y'_p)$ such that  $$G^r(\Pi_{i=1}^pG_i(a'_i)^rf'(b',y'_1(a'_1),\ldots, y'_p(a'_p))=
\tilde{f}(b',\bar{y}'_1(a'_1),\ldots,\bar{ y}'_p(a'_p),G_1(a'_1),\ldots,G_p(a'_p)),$$
increasing again $r$ if necessary. It follows that

 iii) $\tilde{f}(b',\bar{y}'_1(a'_1),\ldots,\bar{ y}'_p(a'_p),G_1(a'_1),\ldots,G_p(a'_p))\in (F).$

\noindent Note that $\tilde f$ depends on $\bar{y}'_i, a'_i, G_i$.     We have the following commutative diagram.

$$
  \begin{xy}\xymatrix{V  \ar[d]^{\theta}\ar[r] & V[Y] \ar[d] \ar[r] & U \ar[d]^{\rho}\\
K\langle x\rangle \ar[r] & K\langle x\rangle[Y] \ar[r] & K'\langle x\rangle \\
K\langle x_{J_i}\rangle \ar[u]\ar[r]  & K\langle x_{J_i}\rangle [Y_i]  \ar[u] \ar[r] &K'\langle x_{J_i}\rangle
 \ar[u] } \end{xy}
  $$

 Next we note that i),ii), i'),ii'),iii) means that the coefficients from $K'$ of $y'_i$, $a'_i$, $b'$, $G_i$, $G$, $F_i$, $F$ are solutions in $K'$ of a certain   system $H$ of polynomial equations over $K$.\\

Let $F_i=\sum_{l_i=1}^{d_i}\bar{F}_{i,l_i}T_i^{l_i}$ where $\bar{F}_{i,l_i}\in K'[x_{J_i}]$ with $\bar{F}_{i,d_i}=1$, $F=\sum_{l=1}^{d}\bar{F}_{l}T^{l}$ where $\bar{F}_{l}\in K'[x]$ with $\bar{F}_d=1$, and $G_i=\sum_{k_i=1}^{m}\bar{G}_{i,k_i}T_i^{k_i}$ where $\bar{G}_{i,k_i}\in K'[x_{J_i}]$, $G=\sum_{k=1}^{m}\bar{G}_{k}T^{k}$ where $\bar{G}_{k}\in K'[x]$ for some $m>>0$. \\

i) Since we have $G^rF_i(a'_i)\in (F)$, where $a'_i=\sum_{s_i=1}^{d-1}a'_{i,s_i}T^{s_i}$, $a'_{i,s_i}\in K'[x]$ it follows that
$$(\sum_{k=1}^{m}\bar{G}_{k}T^{k})^r\sum_{l_i=1}^{d_i}\bar{F}_{i,l_i}(\sum_{s_i=1}^{d}a'_{i,s_i}T^{s_i})^{l_i}=FP$$ for some $P=\sum_{c=1}^{m}{\bar P}_{c}T^{c} \in K'[x][T]$ increasing $m$ if necessary. This  equation has the form:
$$(\sum_{k\leq m,k'\in \mathbb{N}^n,|k'|<\alpha}{\bar{G}}_{k,k'}x^{k'}T^{k})^r(\sum_{l_i<d_i, l'_i\in \mathbb{N}^n, |l'_i|<\alpha}{\bar{F}}_{i,l_i,l'_i}x_{J_i}^{l'_i}(\sum_{s_i< d,s'_i\in \mathbb{N}^n, |s'_i|<\alpha}\bar{a}'_{i,s_i,s'_i}x^{s'_i}T^{s_i})^{l_i})=$$
$$(\sum_{l< d,l'\in \mathbb{N}^n, |l'|<\alpha}{\bar{F}_{l,l'}x^{l'}T^{l})(\sum_{c\leq m,c'\in \mathbb{N}^n, |c'|<\alpha}}{\bar P}_{c,c'}x^{c'}T^{c})$$
for some $\alpha >>0$. Here $\bar{G}_k=\sum_{k'\in \mathbb{N}^n, |k'|<\alpha}\bar{G}_{kk'}x^{k'}$, $\bar{G}_{kk'}\in K'$ and similarly for others.
This gives  a system of polynomial equations  $H_{i}$ over $K$ which  has a solution in $K'$, namely $({\bar{G}}_{k,k'}),({\bar{F}}_{i,l_i,l'_i}),(\bar{a}'_{i,s_i,s'_i}),({\bar{F}}_{l,l'}),({\bar P}_{c,c'})$.

ii) Since $G_i(a'_i)| G^r$ modulo $(F)$,  there exist $Q,L \in K[x][T]$ such that $G_i(a'_i).Q=G^r+F.L$. We may assume that $Q=\sum _{q=1}^{m}{\bar Q}_qT^q$ and $L=\sum _{t=1}^{m}{\bar L}_tT^t$ increasing $m$ if necessary. Then we get:
$$(\sum_{k_i\leq m,k'_i\in \mathbb{N}^n, |k'_i|<\alpha}{\bar{G}}_{i,k_i,k'_i}x^{k'_i}(\sum_{s_i<d,s'_i\in \mathbb{N}^n, |s'_i|<\alpha}\bar{a}'_{i,s_i,s'_i}x^{s'_i}T^{s_i})^{k_i})(\sum_{q\leq m,q'\in \mathbb{N}^n, |q'|<\alpha}{\bar Q}_{q,q'}x^{q'}T^q)=$$
$$(\sum_{k\leq m,k'\in \mathbb{N}^n, |k'|<\alpha}\bar{G}_{k,k'}x^{k'}T^{k})^r+(\sum_{l< d,l'\in \mathbb{N}^n, |l'|<\alpha}\bar{F}_{l,l'}x^{l'}T^{l})(\sum_{t\leq m,t'\in \mathbb{N}^n, |t'|<\alpha}{\bar L}_{t,t'}x^{t'}T^t)$$
increasing $m,\alpha$ if necessary.
This gives  a system of polynomial equations  $H_{ii}$ over $K$ which  has a solution in $K'$, namely
 $({\bar{G}}_{i,k_i,k'_i}),(\bar{a}'_{i,s_i,s'_i}),({\bar Q}_{q,q'}),({\bar{F}}_{l,l'}),({\bar L}_{t,t'}),({\bar{G}}_{k,k'})$.\\

Let $\tilde{F}=\sum_{\tilde{l}=1}^{\tilde{d}}\bar{\tilde{F}}_{\tilde{l}}T^{\tilde{l}}$ where $\bar{\tilde{F}}_{l}\in K[x]$ with $\bar{\tilde{F}}_{\tilde{d}}=1$, $\tilde{G}=\sum_{\tilde{k}=1}^m\bar{\tilde{G}}_{\tilde{k}}T^{\tilde{k}}$ where $\bar{\tilde{G}}_{\tilde{k}}\in K'[x]$.

i') Since we have $G^r\tilde {F}(b')\in (F)$, where $b'=\sum_{\tilde{s}=1}^{d-1}b'_{\tilde{s}}T^{\tilde{s}}$,  we have
$$(\sum_{k=1}^{m}\bar{G}_{k}T^{k})^r\sum_{\tilde{l}=1}^{\tilde{d}}\bar{\tilde{F}}_{\tilde{l}}(\sum_{\tilde{s}=1}^db'_{\tilde{s}}T^{\tilde{s}})^{\tilde{l}}=F\tilde{P}$$ for some $\tilde{P}=\sum_{\tilde{c}=1}^m\bar{\tilde{P}}_{\tilde{c}}T^{\tilde{c}} \in K'[x][T]$ increasing $m$ if necessary. This equation has the form:
$$(\sum_{k\leq m,k'\in \mathbb{N}^n, |k'|<\alpha}\bar{G}_{k,k'}x^{k'}T^{k})^r(\sum_{\tilde{l}=1}^{\tilde{d}}(\sum_{\tilde{l}'\in \mathbb{N}^n, |\tilde{l}'|<\alpha}\bar{\tilde{F}}_{\tilde{l}\tilde{l}'}x^{\tilde{l}'}(\sum_{\tilde{s}< d,\tilde{s}'\in \mathbb{N}^n, |\tilde{s}'|<\alpha}\bar{b}'_{\tilde{s},\tilde{s'}}x^{\tilde{s'}}T^{\tilde{s}})^{\tilde{l}}))=$$
$$(\sum_{l< d,l'\in \mathbb{N}^n,  |l'|<\alpha}\bar{F}_{l,l'}x^{l'}T^{l})(\sum_{\tilde{c}\leq m,\tilde{c'}'\in \mathbb{N}^n, |\tilde{c}'|<\alpha}\bar{\tilde{P}}_{\tilde{c},\tilde{c}'}x^{\tilde{c}'}T^{\tilde{c}}).$$
 Here $\bar{\tilde{P}}_{\tilde{c}}=\sum_{\tilde{c'}'\in \mathbb{N}^n, |\tilde{c}'|<\alpha}\bar{\tilde{P}}_{\tilde{c}\tilde{c}'}x^{\tilde{c}'}$, $\bar{\tilde{P}}_{\tilde{c}\tilde{c}'}\in K'[x]$ and similarly for others.
It gives a system of polynomial equations $H_{i'}$ over $K$ which has a solution in $K'$, namely $(\bar{G}_{k,k'}),(\bar{b}'_{\tilde{s},\tilde{s'}}),(\bar{F}_{l,l'}),(\bar{\tilde{P}}_{\tilde{c},\tilde{c}'})$.

ii') Since $\tilde G(b')| G^r$ modulo $(F)$,  there exist $\tilde{Q},\tilde{L} \in K'[x][T]$ such that $\tilde{G}(b').\tilde{Q}=G^r+F.\tilde{L}$. We may choose $\tilde{Q}=\sum _{\tilde{q}=1}^{m}\bar{\tilde{Q}}_{\tilde{q}}T^{\tilde{q}}$ and $\tilde{L}=\sum_{\tilde{t}=1}^{m}\bar{\tilde{L}}_{\tilde{t}}T^{\tilde{t}}$ increasing $m$ if necessary. Now this equation has the following form:
$$\sum_{\tilde{k}\leq \tilde{m},\tilde{k}'\in \mathbb{N}^n, |\tilde{k}'|<\alpha}\bar{\tilde{G}}_{\tilde{k},\tilde{k}'}x^{\tilde{k}'}(\sum_{\tilde{s}< d,\tilde{s}'\in \mathbb{N}^n, |\tilde{s}'|<\alpha}\bar{b}'_{\tilde{s},\tilde{s}'}x^{\tilde{s}'}T^{\tilde{s}})^{\tilde{k}}.(\sum_{\tilde{q}\leq m,\tilde{q}'\in \mathbb{N}^n, |\tilde{q}'|<\alpha}\bar{\tilde{Q}}_{\tilde{q},\tilde{q}'}x^{\tilde{q}'}T^{\tilde{q}})$$=
$$(\sum_{k\leq m,\tilde{k}'\in \mathbb{N}^n, |\tilde{k}'|<\alpha}\bar{G}_{k,k'}x^{k'}T^{k})^r+(\sum_{l< d,l'\in \mathbb{N}^n, |l'|<\alpha}\bar{F}_{l,l'}x^{l'}T^{l})(\sum_{\tilde{t}\leq m,\tilde{t}'\in \mathbb{N}^n, |\tilde{t}'|<\alpha}\bar{\tilde{L}}_{\tilde{t},\tilde{t}'}x^{\tilde{t}'}T^{\tilde{t}})$$ increasing $\alpha$ if necessary.
This gives a system of polynomial equations $H_{ii'}$ over $K$ which has a solution in $K'$, namely $(\bar{b}'_{\tilde{s},\tilde{s}'}),(\bar{F}_{l,l'}),(\bar{\tilde{L}}_{\tilde{t},\tilde{t}'}),(\bar{G}_{k,k'}),(\bar{\tilde{Q}}_{\tilde{q},\tilde{q}'})$ .\\

 iii) Since $\tilde{f}(b',\bar{y}'_1(a'_1),\ldots,\bar{ y}'_p(a'_p),G_1(a'_1),\ldots,G_p(a'_p))\in (F)$ in $K'[x,T]$, there exists $P=
 \sum_{c\leq m,c'\in \mathbb{N}^n, |c'|<\alpha}{\bar P}_{c,c'}x^{c'}T^{c}\in K'[x,T]$ (increasing $m,\alpha$ if necessary) such that
$$\tilde{f}(b',\bar{y}'_1(a'_1),\ldots,\bar{ y}'_p(a'_p),G_1(a'_1),\ldots,G_p(a'_p))=$$
$$(\sum_{l< d, l'\in \mathbb{N}^n, |l'|<\alpha}{\bar{F}}_{l,l'}x^{l'}T^{l})(\sum_{c\leq m,c'\in \mathbb{N}^n, |c'|<\alpha}{\bar P}_{c,c'}x^{c'}T^{c})$$
 where $y'_i=\sum_{j_i,q_i\in \mathbb{N}^n, |q_i|<\alpha, k_i<d_i}{\bar y}'_{i,k_i,q_i}x^{q_i}T_i^{k_i}$, ${\bar y}'_{i,k_i,q_i}\in K'$.
This gives a system of  polynomial equations $H_{iii}$ over $K$ which has a solution in $K'$, namely $(\bar{G}_{i,k_i,k'_i}),(\bar{F}_{l,l'}),$
$(\bar{b}'_{\tilde{s},\tilde{s'}}),(\bar{a}'_{i,s_i,s'_i})$, $(\bar{\tilde{P}}_{\tilde{c},\tilde{c}'})$,
 $(\bar{y}'_{i,k_i,q_i})$.

A solution of $H=H_i\cup H_{ii}\cup H_{i'}\cup H_{ii'}\cup H_{iii}$ in $K$ will define some elements $\tilde a_i$, $\tilde b$, $ \tilde y_i$ over $K$, some polynomials $\tilde F_i$, $\tilde G_i$,
 $\tilde F'$, $\tilde G'$ over $K$, some etale algebras $\tilde U_i= (K[x_{J_i}][T_i]/(\tilde F_i))_{\tilde G_i}$, $\tilde U= (K[x][T]/(\tilde F'))_{\tilde G'}$ and some maps $\tilde \phi_i:\tilde U_i\to \tilde U$,
 $\tilde \rho_i:\tilde U_i\to K\langle x_{J_i}\rangle $,  $\tilde \rho:\tilde U\to K\langle x\rangle $, and similar $\tilde{\psi}$, $\tilde{\phi}$ which make commutative two diagrams similar as above but written for the case $K=K'$
if we can show that $\tilde{\rho}_i$ and $\tilde{\theta}$ are given by $\tilde{\rho}\tilde{\phi}_i$, resp. $\tilde{\rho}\tilde{\psi}$.

 Suppose that $\rho_i$, $\rho$, $\theta$ are given by  $\rho_i(T_i)=\hat{z}_i=\sum_k \hat{z}_{ik} x^k$, $\rho(T)=\bar{z}=\sum_k \bar{z}_{k} x^k$,  $\theta(\tilde{T})=z'=\sum_k z'_{k} x^k$. Then we have $a'_i(\bar{z})=\hat{z}_i$, $b'(\bar{z})=z'$ and so $(\bar{G}_{i,k_i,k'_i}),(\bar{F}_{l,l'}),$
$(\bar{b}'_{\tilde{s},\tilde{s'}}),(\bar{a}'_{i,s_i,s'_i})$, $(\bar{\tilde{P}}_{\tilde{c},\tilde{c}'})$,
 $(\bar{y}'_{i,k_i,q_i})$,  $(\bar{z}_k)$, $(\hat{z}_k)$, $(\bar{b}'_{s,s'})$, $(z'_k)$ is a solution of a certain system of polynomial equations $(\Lambda_k)_{k\in \mathbb{N}^n}$. A solution of $H\cup (\Lambda_k)_{k\in \mathbb{N}^n}$ will define the maps $\tilde{\rho}$, $\tilde{\rho}_i$, $\tilde{\theta}$, $\tilde{\psi}$, $\tilde{\phi}$ which make indeed commutative the two diagrams above written for $K'=K$.

But $(\Lambda_k)_{k\in \mathbb{N}^n}$ is an infinite set of equations and we have to see that just a finite set of them are actually enough.
 Set $\ord \hat{z}_i=\nu_i$, $\ord \bar{z}=\bar{\nu}$, $\ord z'=\nu'$ and $\mu=\max\{(\nu_i),\bar{\nu},\nu'
\}$. Using the unicity of the Implicit Function Theorem ($V,U,U_i$ are etale) we see that $(\hat{z}_{ik})$, $\bar{z}_k$, $(z'_k)$  are uniquelly defined
 by $(\hat{z}_{ik})_{|k|\leq \mu}$, $(\bar{z}_k)_{|k|\leq \mu}$ $(z'_k)_{|k|\leq \mu}$. Taken $\Lambda=(\Lambda_k)_{|k|\leq\mu}$ we see that the two above diagrams are commutative, that is $a'_i(\bar{z})=\hat{z}_i$, $b'(\bar{z})=z'$ if and only if the solution of $H$ in $K$ is also a solution  of $ \Lambda$, providing that the order of the corresponding $\tilde{\rho}(T)$, $\tilde{\rho}(T_i)$, $\tilde{\theta}(\tilde{T})$ is  $\leq \mu$. Assume that
$\hat{z}_{i,\hat{\tau}_i}\not =0$, $\bar{z}_{\bar{\tau}}\not =0$, $z'_{\tau'}\not =0$ for some $\hat{\tau}_i$, $\bar{\tau}$, $\tau'$ with $|\hat{\tau}_i|=\nu_i$, $|\bar{\tau}|=\bar{\nu}$, $|\tau'|=\nu'$.  Then  $\hat{z}_{i,\hat{\tau}_i}$, $\bar{z}_{\bar{\tau}}$, $z'_{\tau'}$ and their inverses satisfy also the system $\Delta$ of equations $\hat{Z}_{i,\hat{\tau}_i } \hat{Z}_{-1}=1$,  $\bar{Z}_{\bar{\tau} } \bar{Z}_{-1}=1$,  $Z'_{\tau' } Z'_{-1}=1$. Asking for a solution of $H\cup \Lambda\cup \Delta$ in $K$ means to get  the maps $\tilde{\rho}$, $\tilde{\rho}_i$, $\tilde{\theta}$, $\tilde{\psi}$, $\tilde{\phi}$ satisfying the two above commutative  diagrams written for $K'=K$.

  Then $y_i=\tilde\rho_i(\tilde y_i)$ form  a solution of $f$ in $K\langle x\rangle$ which is in $\mathcal{N}$. Unfortunately, $\ord y_i\not=\ord \hat y_i$ in general and to get equality we must choose the solution of $H\cup \Lambda\cup \Delta$ more carefully, that is satisfying also some other system of equations.

Note that there exists a system of polynomials $(\Gamma_{ij})_{j\in \mathbb{N}^n}$ in some variables $Y', Z$ such that $\hat{y}_i=\rho_i(y'_i)=y'_i(T_i=\hat{z})=
\sum_{j\in \mathbb{N}^n}\Gamma_{ij}((\bar{y}'_{i,k_i,q_i}),(\hat{z}_k))x^j$.
Then\\
 $\Gamma_{ij}((\bar{y}'_{i,k_i,q_i}), (\hat{z}_{ik}))=0$ for all $j$ with $|j|<c_i$ and  $\Gamma_{i\gamma}((\bar{y}'_{i,k_i,q_i}), (\hat{z}_{ik}))\not=0$
for some $\gamma$ with $|\gamma|=c_i$. Note that only a finite number
$(\bar{y}'_{i,k_i,q_i}), (\hat{z}_{ik})$, let us say for $q_i,k\in \mathbb{N}^n$,$|q_i|,|k|<\omega$, will enter in $\Gamma_{ij}$ when $|j|\leq c_i$. We may suppose that $\omega\geq \mu$. Then $(\bar{y}'_{i,k_i,q_i})_{k_i\leq d_i,|q_i|<\omega}$,
$ (\hat{z}_{ik})_{|k|<\omega}$ and some $\hat{z}_{i\gamma}''$ is a solution of the system $\Gamma_i$ given by the polynomial equations
$$\Gamma_{ij}((Y'_{i,k_i,q_i})_{k_i\leq d_i,|q_i|<\omega}, (Z_{ik})_{|k|<\omega})=0,\ \ |j|<c_i,$$
$$\Gamma_{i\gamma}((Y'_{i,k_i,q_i})_{k_i\leq d_i,|q_i|<\omega}, (Z_{ik})_{|k|<\omega}) Z''_{i\gamma}=1.$$
Thus
$(\bar{G}_{i,k_i,k'_i})$, $(\bar{F}_{l,l'}),$ $(\bar{G}_{kk'})$, 
$(\bar{F}_{i,l_i,l'_i}),$ $(\bar{G}_{kk'}$
$(\bar{b}'_{\tilde{s},\tilde{s'}}),(\bar{a}'_{i,s_i,s'_i})$, $(\bar{\tilde{P}}_{\tilde{c},\tilde{c}'})$, $(\bar{\tilde{L}}_{\tilde{t}\tilde{t}'})$,  $(\bar{\tilde{Q}}_{\tilde{q}\tilde{q}'})$
 $(\bar{y}'_{i,k_i,q_i})_{|q_i|\leq \omega}$,
 $(\hat{z}_{ik})_{|k|<\omega}$, $(\hat{z}''_{i\gamma})$, $(\bar{z}_k)$,  $(z'_k)$, $\hat{z}_{-1}$,$\bar{z}_{-1}$, $z'_{-1}$  is a solution of $H\cup \Lambda\cup \Delta\cup \Gamma$, $\Gamma=(\Gamma_i)_{i\in [p]}$ in $K'$.
Choosing a solution of $H\cup \Lambda\cup \Delta\cup \Gamma$
in $K$  we see that $y_i=\tilde{\rho}_i(y'_i)$ satisfies also  $\ord y_i=\ord \hat y_i$, $i\in [p]$.
\hfill\ \end{proof}

\begin{Corollary}\label{corloc.}
Let $K$ be an algebraically closed field, $K \subset K'$  a  field extension and $x=(x_1, \ldots , x_n)$. Let $J_i$, $i\in [p]$ be  subsets of $[n]$, $x_{J_i}=(x_k)_{k\in J_i}$ and $A_i=K\langle x_{J_i}\rangle$, resp. $A'_i=K'\langle x_{J_i}\rangle$, $i\in [p]$ be the algebraic power series in $x_{J_i}$ over $K$ resp. $K'$. Set
 $\mathcal{N}= A_1 \times \cdots \times A_p$ and
  $\mathcal{N'}=A'_1 \times \cdots \times A'_p$. Let $f$ be a system of polynomials from $K\langle x\rangle [Y]$, $Y=(Y_1,\ldots,Y_p)$, and ${\hat y}\in \mathcal{N'}$,  such that $f({\hat y})=0$.
Then there exist $y\in \mathcal{N}$  such that $f(y)=0$ and   $\ord y_i=\ord \hat y_i$ for $i\in [p]$.
 \end{Corollary}
For the proof note that a field extension of an algebraically closed  field is  an algebraically pure  field extension (see \cite[Corollary 1.8]{P0}).

\section{Ultrapower and Nested Strong Artin Approximation}
A {\em filter} on $\mathbb{N}$ is a non-empty family $D$ of subsets of $\mathbb{N}$ satisfying
\begin{enumerate}
  \item $\emptyset \notin D$,
  \item if $s,t \in D$ then $s \cap t \in D$,
  \item if $s \in D$ and $s\subset t \subset \mathbb{N}$ then $t \in D$.
\end{enumerate}
An {\em ultrafilter} on $\mathbb{N}$ is a maximal filter in the set of filters on $\mathbb{N}$ with respect to inclusion. A filter $D$ is an ultrafilter if and only if $\mathbb{N}\setminus s \in D$ for each subset $s \subset \mathbb{N}$ such that $s \notin D$.
It follows that $s\cup t=\mathbb{N}$ implies $s\in D$ or $t\in D$ because $D$ is an ultrafilter.
 An ultrafilter is called nonprincipal if there exist no $r$ such that $D=\{s|r \in s \subset \mathbb{N}\}$. More precisely, an ultrafilter is nonprincipal if and only if it contains the filter of all cofinite sets of $\mathbb{N}$.

Let $(A_i)_{i \in \mathbb{N}}$ be a family of rings and $\mathcal{I}=\{(x_i)_{i \in \mathbb{N}}\in \prod_{i \in \mathbb{N}}A_i|\{i|x_i=0\}\in D\}$ an ideal in $\prod_{i \in \mathbb{N}}A_i$. We call the factor ring $\prod_{i \in \mathbb{N}}A_i/\mathcal{I}$ the {\em ultraproduct} of the family $(A_i)_{i \in \mathbb{N}}$ with respect to the  ultrafilter $D$. If $A_i=A$ for all $i\in \mathbb{N}$, then denote $A^*=\prod_{i \in \mathbb{N}}A_i/\mathcal{I}$ and call $A^*$ the {\em ultrapower} of $A$ with repect to the ultrafilter $D$. An element $a \in A^*$ has the form $[(a_i)_{i \in \mathbb{N}}]$, $a_i \in A$,  where ``[ ]" means the class modulo $\mathcal{I}$. We recall below some properties of the ultrapower given in \cite{BDLV}, \cite{P0}, \cite{P}, \cite{P1}.

If $A$ is a local ring with $\mm$ its maximal ideal, then $\mm^*=\{[(x_i)_{i \in \mathbb{N}}]| \{i_i\in \mathbb{N}|x_i \in \mm\}\in D\}$ is a maximal ideal in $A^*$ and the unique one. If $K$ is the residue  field of $A$ then the residue field of $A^*$ is the ultrapower $K^*$ of $K$ with respect to the ultrafilter $D$.

\begin{Theorem} \label{t2}
 Let $K$ be a field,  $ A=K\langle x\rangle$, $x=(x_1,\ldots,x_n)$,  $f=(f_1,\ldots,f_r)\in K\langle x\rangle[Y]^r$, $Y=(Y_1,\ldots,Y_p)$  and  $0\leq s_1\leq \ldots \leq s_p\leq n$ be some non-negative integers.
Then there exists a map $\nu:\mathbb{N}^p\to \mathbb{N} $ such that if $y'=(y'_1,\ldots,y'_p)$,  $y'_i\in K[x_1,\ldots,x_{s_i}]$, $i\in [p]$ satisfies  $f(y')\equiv 0$ modulo $(x)^{\nu(c)}$ for some $c=(c_1,\ldots,c_p)\in \mathbb{N}^p$ and $\ord y'_i=c_i$, $i\in [p]$ then there exists $y_i\in K\langle x_1,\ldots,x_{s_i}\rangle$ for all $ i\in [p]$ such that $y=(y_1,\ldots,y_p)$ is a solution of $f$ in $A$ and  $\ord y_i=c_i$ for all $i\in [p]$.
\end{Theorem}

\begin{proof} We will show that given $c=(c_1,\ldots,c_p)\in \mathbb{N}^p$ there exists an integer $k_c\in \mathbb{N}$ such that if    $y'=(y'_1,\ldots,y'_p)$,  $y'_i\in K[x_1,\ldots,x_{s_i}]$ satisfies  $f(y')\equiv 0$ modulo $(x)^{k_c}$  and $\ord y'_i=c_i$, $i\in [p]$ then there exists $y_i\in K\langle x_1,\ldots,x_{s_i}\rangle$ for all $ i\in [p]$ such that $y=(y_1,\ldots,y_p)$ is a solution of $f$ in $A$ and  $\ord y_i=c_i$ for all $i\in [p]$.

Assume that this is false. Thus there exists $c$ such that for all $k\in \mathbb{N}$ there exists $y'^{(k)}=(y'^{(k)}_1\ldots,y'^{(k)}_p)$,  $y'^{(k)}_i\in K[x_1,\ldots,x_{s_i}]$ with  $f(y'^{(k)})\equiv 0$ modulo $(x)^{k}$ and $\ord y'^{(k)}_i=c_i$, $i\in [p]$,
but there exists no $y=(y_1,\ldots,y_p)$, $y_i\in K\langle x_1,\ldots,x_{s_i}\rangle$ for all $ i\in [p]$ which is a solution of $f$ in $A$ and  $\ord y_i=c_i$ for all $i\in [p]$.

Let $D$ be a nonprincipal ultrafilter on $\mathbb{N}$ and $(K\langle x\rangle )^*$ be the ultrapower of $K\langle x\rangle$ with respect to $D$.
Take ${y'}_i^*=[({y'}_i^{(k)})_{k\in \mathbb{N}}] \in (K[ x_1, \ldots , x_{s_i}])^* \subset  (K\langle x_1, \ldots , x_{s_i}\rangle)^* \subset (K\langle x\rangle )^*$. Consider the canonical surjection  
$$\eta_i:(K\langle x_1, \ldots , x_{s_i}\rangle)^* \to (K\langle x_1, \ldots , x_{s_i}\rangle)^* /\bigcap_{l \in \mathbb{N}}(x_1, \ldots , x_{s_i})^l,$$ and set $ \hat{{y'}_i^*}=\eta_i({y'}_i^*) $. We have the following commutative diagram:

$$
  \begin{xy}\xymatrix{(K\langle x_1, \ldots , x_{s_i}\rangle)^*  \ar[d]^{\eta_i}\ar[r] & (K\langle x \rangle)^* \ar[d] \\
( K\langle x_1, \ldots , x_{s_i}\rangle)^* /\bigcap_{l \in \mathbb{N}}(x_1, \ldots , x_{s_i})^l\ar[r] &(K\langle x \rangle)^*/\bigcap_{l \in \mathbb{N}}(x)^l
 } \end{xy}
  $$

By \cite[Proposition 2.3]{P0} (see also \cite{P}, \cite[Theorem 2.5]{P1}, \cite[Lemma 3.3.2]{R}) we know that
 $$(K\langle x_1, \ldots , x_{s_i}\rangle)^* /\bigcap_{l \in \mathbb{N}}(x_1, \ldots , x_{s_i})^l\cong K^*[[x_1, \ldots , x_{s_i}]]$$ and similarly
 $$(K\langle x \rangle)^*/\bigcap_{l \in \mathbb{N}}(x)^l \cong K^*[[x]].$$ Thus we may consider $\hat{{y'}^*_i}\in  K^*[[ x_1,\ldots,x_{s_i}]]\subset  K^*[[ x]]$ and note that\\ $\hat{{y'}^*} =(\hat{{y'}_1^*},\ldots,\hat{{y'}_p^*})$ is a solution of $f$ in $K^*[[x]]$. Indeed,
 we have  $f(y'^{(k)})\equiv 0$ modulo $(x)^{k}$ for all $k$. Fix $t\in\mathbb{N}$. we have $f(y'^{(k)})\in (x)^{t}$ for all $k \geq t$ and so $\{k\in \mathbb{N}:f(y'^{(k)})\in (x)^{t}\}\in D$ because $ D$ contains cofinite sets, being a nonprincipal ultrafilter.  It follows that $f({y'}^*)\in (x)^t( K\langle x_1, \ldots , x_n\rangle)^*$ for all $t$. So we have $f({y'}^*)\in \bigcap_{t \in \mathbb{N}}(x)^t( K\langle x_1, \ldots , x_n\rangle)^*$ and hence $\hat{{y'}^*} =(\hat{{y'}_1^*},\ldots,\hat{{y'}_p^*})$ gives a solution of $f$ in $K^*[[x]]$.

 We claim that $\ord \hat{{y'}^*}_i=c_i$. Indeed, set
  $s_{ij}=\{k\in \mathbb{N}:{y'}^{(k)}_{ij}\not =0\}$.
   By assumption $\ord {y'}^{(k)}_i=c_i$, $i\in [p]$ and so $\cup_{j\in \mathbb{N}^n, |j|=c_i}s_{ij}=\mathbb{N}$ which implies
 $s_{ij_i}\in D$ for some $j_i$ because $D$ is an ultrafilter. It follows that ${y'}^*_{ij_i}\not =0$ and so $\hat{{y'}^*}_{ij_i}\not =0$, which shows our claim.

 By  Nested Artin Approximation Property \cite[Theorem 3.7]{P} (see also \cite[Corollary 3.7]{P1}, \cite[Theorem 5.2.1]{R}) this implies that for all $u \in \mathbb{N}$ there exists a solution $\tilde{y}=(\tilde{y}_1, \ldots , \tilde{y}_p)\in K^*\langle x \rangle$ of $f$ with $\tilde{y}_i \in  K^*\langle x_1, \ldots , x_{s_i}\rangle$ and $\tilde{y}_i \equiv \hat{{y'}_i^*} $ mod $x^u$ for all $i$. Take $u > c_i$ for all $i$. This implies that ord $\tilde{y}_i$= ord $\hat{{y'}_i^*}$, $i\in [p]$. Since  $K \to K^*$ is algebraically pure we get by  Theorem \ref{t} a $y''$ in $K\langle x \rangle$ such that $y''_i \in K\langle x_1, \ldots , x_{s_i}\rangle$ with ord $y''_i$= ord $\tilde{y}_i$ and $f(y'')=0$ , a contradiction.
\hfill\ \end{proof}

\begin{Remark} {\em The above theorem does not hold if $A$ is the complex convergent power series ring in $x$. Indeed, Gabrielov's example (see \cite{Ga}, or \cite[Example 5.3.1]{R}) gives a nested formal solution of a certain polynomial equation $f$, which has no nested convergent solutions. Clearly, the formal solution defines some polynomials $y^{(k)}\in \mathbb{C}[x]$ such that $f(y^{(k)}\equiv 0$ mod $(x)^k$ for all $k\in \mathbb{N}$. Thus Question \ref{q4} has a negative answer for $A$.}
\end{Remark}

\begin{Remark} (Rond)\label{r1} {\em In fact it is true a stronger result,  namely:
 There exist a   map $\nu':\mathbb{N}^p\to \mathbb{N} $ such that if $y'=(y'_1,\ldots,y'_p)$,  $y'_i\in K[x_1,\ldots,x_{s_i}]$, $i\in [p]$ satisfies  $f(y')\equiv 0$ modulo $(x)^{\nu'(c)}$ for some $c\in \mathbb{N}$  then there exists $y_i\in K\langle x_1,\ldots,x_{s_i}\rangle$ for all $ i\in [p]$ such that $y=(y_1,\ldots,y_p)$ is a solution of $f$ in $A$ and  $y_i\equiv y'_i$ modulo $(x)^c$ for all $i\in [p]$.

This is stated in \cite[Corollary 5.2.4]{R}. Its proof is done  in a particular case \cite[Theorem 4.2]{BDLV} (see also \cite[Remark, page 199]{BDLV}) but  the general case follows in the same way using the Nested Approximation property \cite[Theorem 3.7]{P}.}
\end{Remark}

{\bf Acknowledgements} The first author  gratefully acknowledges the support from the ASSMS GC. University Lahore, for arranging her visit to Bucharest, Romania and she is also grateful to the Simion Stoilow Institute of the Mathematics of the Romanian Academy for inviting her.


\vskip 0.5 cm


\begin{thebibliography}{99}
\bibitem{A0} M.\ Artin, {\em On the solutions of analytic equations}, Invent. Math., {\bf 5}, (1968), 277-291.

\bibitem{A} M.\ Artin, {\em Algebraic approximation of structures over complete local rings}, Publ. Math. IHES, {\bf 36}, (1969), 23-58.
\bibitem{A1} M.\ Artin, {\em Constructions techniques for algebraic spaces}, Actes Congres. Intern. Math., t 1, (1970), 419-423-291.

\bibitem{BDLV} J.\ Becker, J.\ Denef, L.\ Lipshitz, L.\ van den Dries, {\em Ultraproducts
and approximation in local rings I}, Invent. Math., {\bf 51}, (1979), 189-203.

\bibitem{BM} E.\ Bierstone, P.\ Milman, {\em Invariant solutions of of analytic equations}, Enseign. Math. {\bf 25}, (1979), 115-130.
\bibitem{BPR} M.\ Bilski, A.\ Parusinski, G.\ Rond, {\em Local topological algebraicity of analytic function germs}, J.\  Algebraic Geom., {\bf 26}, (2017), 177-197.  
\bibitem{CPR} F.\ J.\ Castro-Jimen\'ez, D.\ Popescu, G.\ Rond, {\em Linear nested Artin approximation for algebraic power series},  2016, arXiv:AC/1511.09275v4.

 \bibitem{Ga} A.\ M.\ Gabrielov, {\em The formal relations between analytic functions}, Funkcional. Anal. i Prilovzen, {\bf 5}, (1971), 64-65.

\bibitem{Gra} H.\ Grauert, {\em \"Uber die Deformation isolierter Singularit\"aten analytischer Mengen}, Invent. Math., {\bf 15}, (1972), 171-198.

\bibitem{Gr} M.\ Greenberg, {\em Rational points in henselian discrete valuation rings}, Publ. Math. IHES, {\bf 31}, (1966), 59-64.
\bibitem{K}  H.\ Kurke,  T.\ Mostowski, G.\ Pfister, D.\ Popescu, M.\ Roczen,  {\em Die Approximationseigenschaft
lokaler Ringe}, Springer Lect. Notes in Math., {\bf 634}, Springer-Verlag, Berlin-New
York, (1978).

\bibitem{Mir} N.\ Mir, {\em Artin Appoximation Theorems and Cauchy-Riemann Geometry}, Methods Appl. Anal., {\bf 21}, (2014), 481-502.
\bibitem{Mo} T.\ Mostowski, {\em Topological equivalence between analytic and algebraic sets}, Bull. Polish Acad. Sci. Math., {\bf 32}, (1984), no. 7-8, 
\bibitem{Pi} G.\ Picavet, {\em Algebraically flat on projective algebras}, J., Pure and Appl. Alg., {\bf 174} (2), (2002), 163-185.



\bibitem{PP} G.\ Pfister, D.\ Popescu, {\em  Die strenge Approximationseigenschaft lokaler Ringe}, Inventiones Math.
{\bf 30}, (1975), 145-174.

\bibitem{P0} D.\ Popescu, {\em  Algebraically pure morphisms}, Rev.Roum.Math.Pures et Appl., {\bf 24}, (1979),
947-977.

 \bibitem{P} D.\ Popescu, {\em General Neron Desingularization and approximation}, Nagoya Math. J., {\bf 104}, (1986), 85-115.
 \bibitem{P1} D.\ Popescu, {\em Artin Approximation}, in "Handbook of Algebra", vol. 2, Ed. M. Hazewinkel, Elsevier, 2000, 321-355.

\bibitem{P2} D.\ Popescu, {\em Artin approximation property and the General Neron Desingularization},  Revue Roum. Math. Pures et Appl., {\bf 62}, (2017), 171-189,  arXiv:AC/1511.06967.

\bibitem{R}  G.\ Rond,  {\em Artin approximation}, arXiv:AC/1506.04717.

\end{thebibliography}
\end{document}